\documentclass[11pt]{amsart}

\usepackage{amssymb, amsmath, amsthm}

 \newtheorem{thm}{Theorem}[section]
 \newtheorem{cor}[thm]{Corollary}
 \newtheorem{lem}[thm]{Lemma}
 \newtheorem{prop}[thm]{Proposition}

    \theoremstyle{definition}
 \newtheorem{defn}[thm]{Definition}

 \newtheorem{exam}[thm]{Example}
 \theoremstyle{remark}
 \newtheorem{rem}[thm]{Remark}
\numberwithin{equation}{section}

\begin{document}

\title[Hom-Post-Lie modules, $\mathcal{O}$-operator and some functors on Hom-algebras]{Hom-Post-Lie modules, 
$\mathcal{O}$-operator and some functors on Hom-algebras}

\author{Ibrahima BAKAYOKO }

\address{D\'epartement de Math\'ematiques, UJNK/Centre Universitaire de N'Z\'er\'ekor\'e, BP : 50, N'Z\'er\'ekor\'e, Guinea}
\email{ibrahimabakayoko27@gmail.com}



\date{October 09, 2016.}


\begin{abstract}
The aim of this paper is to 
 study modules over Hom-post-Lie algebras and give some contructions and various twistings i.e. we show that modules
over Hom-post-Lie algebras are close by twisting either by Hom-post-Lie algebra endomorphisms or module structure maps.
Given a type of Hom-algebra $A$,  an $A$-bimodule $M$ and an $\mathcal{O}$-operator $T: A\rightarrow M$ ,  we give construction
 of another Hom-algebra structure on $M$. 
\end{abstract}

\maketitle
\section{Introduction}
Hom-algebraic structures are algebras where the identities defining the structure are stwisted by a homomorphism. They  were appeared for the
 first time in the works of Aizawa N. and Sato H. \cite{NH} as a generalization of Lie algebras. They have intensively
 investigated in the literature recently. Hom-type algebraic structures of many classical structures were studied as Hom-associative algebras,
Hom-Lie admissible algebras and more general G-Hom-associative algebras \cite{MS}, n-ary Hom-Nambu-Lie algebras \cite{FSA}, 
Rota-Baxter operators on pre-Lie superalgebras and beyond \cite{A}, Ternary Leibniz color algebras and beyond \cite{BIA},
Rota-Baxter Hom-Lie admissible algebras  \cite{MD}, Non-Commutative Ternary Nambu-Poisson Algebras and Ternary Hom-
Nambu-Poisson Algebras \cite{MA},  Laplacian of Hom-Lie quasi-bialgebras  \cite{BI5}, Hom-Novikov color algebras\cite{B3}, Hom-$\mathcal{O}$-operators
and Hom-Yang-Baxter equations\cite{YL}, Some remarks on Hom-path modules and Hom-path algebras \cite{SH}. And the refernces of these papers.

Post-Lie algebras first arise form the work of Bruno Vallette \cite{BV} in 2007 through the purely operadic technique of Koszul dualization.
In \cite{ML}, it shown that they also arise naturally from differential geometry of homogeneous spaces and Klein geometries, topics that are 
closely related to Cartan's method of moving frames. The universal
enveloping algebras of post-Lie algebras and the free post-Lie algebra are studied. Some examples and related structures are given.

Motivated by the generalization of (the operator form of) the classical Yang-Baxter equation in the Lie algebra \cite{SE}, \cite{KU}, 
Kupershmidt  \cite{KU} introduced the notion of $\mathcal{O}$-operator for a Lie algebra. But this can be traced to Bodermann \cite{BO1}
 in the study of integrable systems. The   $\mathcal{O}$-operator for associative algebra has been done by \cite{BGN} and independantly
 in \cite{U} under the name of generalized Rota-Baxter operator. In \cite{CB}, C. Bai introduced $\mathcal{O}$-operator for Loday algebras, 
studied the  relationship between these algebras, as well as their connection with the analogues of Yang-Baxter equation for these algebras.
From several different motivations, $L$-dendriform algebra are introduced by \cite{CBX} are the Lie analogue of dendriform algebras.

The purpose of this paper is to study Hom-post-Lie modules and establish some functors via $\mathcal{O}$-operators. The paper is organized as 
follows. In section 2, we recall basic definitions and properties of Hom-associative algebras, Hom-Lie algebras, Hom-preLie algebras and
 Hom-dendriform algebras. In section 3, we introduce modules over Hom-post-Lie algebras and prove that twisting a  post-Hom-Lie module structure
 map by an endomorphism of Hom-post-Lie algebra  or a linear vector espace we get another one. We show that one can obtain Hom-post-Lie modules
 from only a given multiplicative Hom-post-Lie algebra in a non-trivial sense. We also prove that the direct sum and the tensor product of two 
Hom-post-Lie modules  is also a Hom-post-Lie module. In section 4, we given some functors by using $\mathcal{O}$-operators ; constructions of 
Hom-preLie algebras, Hom-dendriform algebras,  Hom-$L$-dendriform algebras from Hom-associative algebra. Constructions of Hom-preLie algebras
 from Hom-Lie algebras are done. Finally, construction of Hom-$L$-dendriform algebras from Hom-preLie algebras are also given.

Let us fix some notations and conventions :

i) We will write $a_i^je_j$ (Einstein convention) instead of $\sum_j a_i^je_j$ i.e. we will omit the summation symbol.

ii) We will denote $l_\ast(x)$ and $R_\ast(y)$ the left and the right multiplication, that is
$$l_\ast(x)y=x\ast y,\;\;\mbox{and}\;\;R_\ast(y)x=x\ast y.$$

iii) Throughout this paper, all vector spaces are assumed to be over a field $\mathbb{K}$ of characteristic different from 2.

\section{Hom-associative, Hom-(pre)Lie and Hom-dendriform algebras}
In this subsection, we recall some basic definitions.
\begin{defn}
 By a Hom-algebra we mean a triple $(A, [\cdot, \cdot], \alpha)$ in which $A$ is a vector space, $[\cdot, \cdot] : A\otimes A \rightarrow A$ is a
 bilinear map (the multiplication) and $\alpha : A\rightarrow A$ is a linear map (the twisting map).

If in addition,  $\alpha\circ[\cdot, \cdot]=[\cdot, \cdot]\circ(\alpha\otimes \alpha)$, then the Hom-algebra
 $(A, [\cdot, \cdot], \alpha)$ is said to be multiplicative.

 A morphism $f : (A, [\cdot, \cdot], \alpha)\rightarrow (A', [\cdot, \cdot]', \alpha')$ of Hom-algebras is a linear map $f$ of the 
underlying vector spaces such that $f\circ\alpha=\alpha'\circ f$ and $[\cdot, \cdot]'\circ(f\otimes f)=f\circ[\cdot, \cdot]$.
\end{defn}
\begin{rem}
 If $(A, [\cdot, \cdot])$ is a non-necessarily associative algebra in the usual sense, we
also regard it as the Hom-algebra $(A, [\cdot, \cdot], Id_A)$ with identity twisting map.
\end{rem}
\begin{defn}
 Let $(A, [\cdot, \cdot], \alpha)$ be a Hom-algebra.
The Hom-associator of $A$ is the trilinear map $as_\alpha : A^{\otimes3}\rightarrow A$ defined as
\begin{eqnarray}
 as_\alpha=[\cdot, \cdot]\circ([\cdot, \cdot]\otimes\alpha-\alpha\otimes[\cdot, \cdot]).\nonumber
\end{eqnarray}
\end{defn}
\begin{defn}
 A Hom-associative algebra is a triple $(A,\cdot,\alpha)$ consisting of a linear space $A$, 
a $\mathbb{K}$-bilinear map $\cdot :A\times A\longrightarrow A$ and a linear map $\alpha :A\longrightarrow A $ satisfying 
\begin{eqnarray}
as_\alpha(x, y, z)=0\quad\mbox{(Hom-associativity)},
\end{eqnarray}
for all $x, y, z\in A$.
\end{defn}

\begin{defn}
 A Hom-module is a pair $(M,\alpha_M)$ in which $M$ is a vector space and $\alpha_M: M\longrightarrow M$ is a linear map.
\end{defn}
\begin{defn}
Let $(A, \cdot, \alpha)$ be a Hom-associative algebra and $(M,\beta)$ be a Hom-module. Let $l, r : A\rightarrow\mathcal{G}l(M)$ be two linear maps.
The triple $(M, l, r)$ is called an $A$-bimodule if for all $x, y\in A$,
  \begin{eqnarray}
   l(x\cdot y)\beta=l(\alpha(x))l(y),\quad r(\alpha(y))l(x)=l(\alpha(x))r(y), \quad r(\alpha(y))r(x)=r(x\cdot y)\beta. \label{ahm}
  \end{eqnarray}
\end{defn}

\begin{defn}
A Hom-Lie algebra is a triple $(V, [\cdot, \cdot], \alpha)$ consisting of a linear space $V$, a bilinear map
 $[\cdot, \cdot] :V\times V\longrightarrow V$ and a linear map $\alpha :V\longrightarrow V$ satisfying 
 \begin{eqnarray}
  [x,y]=-[y,x],\qquad \qquad\qquad\qquad\qquad\mbox{(skew-symmetry)}\\
{[\alpha(x), [y, z]]+[\alpha(y), [z, x]]+[\alpha(z), [x, y]]=0},\quad \mbox{(Hom-Jacobi identity)}
 \end{eqnarray}
for all $x, y, z\in V$.
\end{defn}
 The product $[., .]$ is called the Hom-Lie bracket.
\begin{defn}
 A Hom-algebra  $(A, \cdot, \alpha)$ is said to be a Hom-Lie admissible algebra if, for any elements
$x, y\in A$, the bracket $[-, -] : A\times A\rightarrow A$ defined
by $$[x, y]=x\cdot y-y\cdot x$$
satisfies the Hom-Jacobi identity.
\end{defn}

\begin{exam}\label{Am1}
Any Hom-associative algebra is Hom-Lie admissible. That is if $(A,\cdot,\alpha)$ be a Hom-associative algebra, 
then $(A, [\cdot,\cdot],\alpha)$ is a Hom-Lie algebra, where $[x,y]=x\cdot y-y\cdot x$, for all $x, y\in A$.
\end{exam}
\begin{defn}\cite{D2.}
 Let $(L, [\cdot,\cdot], \alpha) $ be a Hom-Lie algebra and $(V, \beta)$ be a Hom-module. 
An $L$-module on $V$ consists of a ${\bf K}$-linear map $l : L\rightarrow\mathcal{G}l(V)$ such that for any $v\in V,  x, y\in L$,
\begin{eqnarray}
\beta l(x)&=&l(\alpha(x))\beta\label{exoo2}\\
l([x, y])\beta&=&l(\alpha(x))l(y)-l(\alpha(y))l(x)\;\;\; \label{exoo}
\end{eqnarray}
\end{defn}
\begin{rem}
 When $\beta=Id_V$ and $\alpha=Id_L$, we recover the definition of Lie modules \cite{JEH}.
\end{rem}
\begin{prop}
Let $(A, \cdot, \alpha)$ be a Hom-associative algebra and $(V, l, r, \beta)$ be an $A$-bimodule. Then $(V, l-r, \beta)$ is a module
over the Lie algebra associated to $A$. 
\end{prop}
\begin{defn}
A Hom-algebra $(S, \cdot, \alpha)$ is called a Hom-preLie or Hom-left-symmetric algebra if the following  {\it Hom-left-symmetric identity}
\begin{eqnarray}
as_\alpha(x, y, z)=as_\alpha(y, x, z),
\end{eqnarray}
holds.
Or equivalently,
\begin{eqnarray}
(x\cdot y)\cdot\alpha(z)-\alpha(x)\cdot(y\cdot z)=(y\cdot x)\cdot\alpha(z)-\alpha(y)\cdot(x\cdot z)
\end{eqnarray}
is satisfied for all $x, y, z\in S$.
\end{defn}
 Recall that any Hom-preLie algebra is Hom-Lie admissible. More precisely we have the following Lemma.
\begin{lem}\cite{AS1}\label{pltol}
 Let $(A, \cdot, \alpha)$ be a Hom-preLie algebra. Then the bracket $[-, -]: A\times A\rightarrow A$ defines by
$$[x, y]=x\cdot y-y\cdot x$$
gives a Hom-Lie algebra structure to $A$.
\end{lem}
\begin{defn}
Let $(A, \cdot, \alpha)$ be a Hom-preLie algebra and $(M,\beta)$ be a Hom-module. Let $l, r : A\rightarrow\mathcal{G}l(M)$ be two linear maps.
$(M, l, r)$ is called a bimodule over $(A, \cdot, \alpha)$ if 
 \begin{eqnarray}
  l(x\cdot y)\beta-l(\alpha(x))l(y)&=&l(y\cdot x)\beta-l(\alpha(x))l(x),\\
l(\alpha(x))r(y)-r(\alpha(y))l(x)&=&r(x\cdot y)\beta-r(\alpha(y))r(x),\label{hlsm2}
 \end{eqnarray}
for any $x, y\in A$.
\end{defn}
\begin{defn}\label{tdd}
 A  Hom-dendriform algebra is a quadruple $(A, \dashv, \vdash, \alpha)$ consisting of a  vector space $A$,
 three bilinear maps  $\dashv, \vdash, \cdot : A\otimes A\rightarrow A$ and  a linear map $\alpha : A\rightarrow A$ satisfing 
\begin{eqnarray}
(x\dashv y)\dashv\alpha(z)&=&\alpha(x)\dashv(y\dashv z+y\vdash z),\label{t1}\\
(x\vdash y)\dashv\alpha(z)&=&\alpha(x)\vdash(y\dashv z),\label{t2}\\
\alpha(x)\vdash(y\vdash z)&=&(x\dashv y+x\vdash y)\vdash\alpha(z),\label{t3}
\end{eqnarray}
for $x, y, z\in A$.
\end{defn}

\begin{defn}
A Hom-associative Rota-Baxter algebra of weight $\lambda$ is a Hom-associative
algebra $(A, \cdot, \alpha)$ together with a linear self-map $R : A\rightarrow A$ that satisfies the identities
\begin{eqnarray}
R(x)\cdot R(y) &=& R\Big(R(x)\cdot y + x \cdot R(y) +\lambda x\cdot y\Big)
\end{eqnarray}
\end{defn}
\begin{exam}
First, recall that an $\varepsilon$-Hom-bialgebra \cite{DYO} or infinitesimal Hom-bialgebras is a quadruple $(A, \mu, \Delta, \alpha)$ consisting of a linear space $A$, 
a {\bf K}-bilinear map $\mu :A\times A\longrightarrow A$ and 
 linear space maps  $\Delta:A\longrightarrow A\otimes A$ and $\alpha :A\longrightarrow A $ satisfying 
\begin{eqnarray}
 \alpha(a)\cdot(b\cdot c)&=&(a\cdot b)\cdot\alpha(c)\quad\mbox{(Hom-associativity)},\label{had}\\
(\alpha\otimes\Delta)\circ\Delta&=&(\Delta\otimes\alpha)\otimes \Delta \quad\mbox{(Hom-coassociativity), and} \label{hca} \\
 \Delta(a\cdot b)&=&\alpha(a)\cdot b_1\otimes\alpha(b_2)+\alpha(a_1)\otimes a_2\cdot\alpha(b). \label{hcc}
 \end{eqnarray}
for any $a, b\in A$, where $a\cdot b=\mu(a, b)$.
 In the Sweedler's notation,  the Hom-coassociativity means that
$$\sum \alpha(a_1)\otimes a_{21}\otimes a_{22}=\sum a_{11}\otimes a_{12}\otimes\alpha(a_2).$$
Now let us introduce the notion of {\it bicentroid} for an $\varepsilon$-Hom-bialgebra. A linear map $\alpha : A\rightarrow A$ on a Hom-bialgebra 
$(A, \mu, \Delta, \alpha)$ is called
a {\it bicentroid} if it is both a {\it centroid} that is
\begin{eqnarray}
 \alpha(a)\cdot b&=&\alpha(a\cdot b)=a\cdot\alpha(b), \label{cent}
\end{eqnarray}
and a {\it cocentroid} that is
\begin{eqnarray}
\alpha(a_1)\otimes a_2 &=&(\alpha(a))_1\otimes(\alpha(a))_2=a_1\otimes \alpha(a_2).\label{cocent}
\end{eqnarray}
At the moment, let $(A, \mu, \Delta, \alpha)$ be an $\varepsilon$-Hom-bialgebra in which $\alpha$ is an {\it involutive bicentroid}
 and consider $End_\alpha$ as the set of endomorphism of $A$ that commute with $\alpha$. It is clear that $(End_\alpha, \circ, \gamma)$ is a 
Hom-associative algebra, where $\circ$ is the composition low of maps and $\gamma(f)=\alpha\circ f, f\in End_\alpha$. Define on $End_\alpha$ the 
linear operator $R$ by
$$R(f)(a)=(\alpha\ast f)(a)=(\mu(\alpha\otimes f)\Delta)(a).$$
Then $R$ is a Baxter operator on $End_\alpha$.\\ In fact, for any $a\in A$,
 $$R(f)(a)=(\alpha *f)(a)=\mu(\alpha\otimes f)\Delta(a)=\alpha(a_1)\cdot f(a_2).$$
Then by (\ref{hcc}),
\begin{eqnarray}
 \Delta(R(f)(a))
&=&\alpha^2(a_1)\cdot (f(a_2))_1\otimes \alpha((f(a_2))_2)+\alpha((\alpha(a_1))_1)\otimes(\alpha(a_1))_2\cdot\alpha(f(a_2))\nonumber\\
&=&a_1\cdot (f(a_2))_1\otimes \alpha((f(a_2))_2)+\alpha((\alpha(a_1))_1)\otimes(\alpha(a_1))_2\cdot\alpha(f(a_2)).\nonumber
\end{eqnarray}
It follows that,
\begin{eqnarray}
 (R(g)\circ R(f))(a)
&=&\mu(\alpha\otimes g)\Delta(R(f)(a))\nonumber\\
&=&\mu(\alpha\otimes g)\Big(a_1\cdot (f(a_2))_1\otimes \alpha(f(a_2))_2\nonumber\\
&&+\alpha((\alpha(a_1))_1)\otimes(\alpha(a_1))_2\cdot\alpha(f(a_2))\Big) \nonumber\\
&=&(\alpha(a_1)\cdot (f(a_2))_1)\cdot g(\alpha((f(a_2))_2))
+(\alpha(a_1))_1\cdot g((\alpha(a_1))_2\cdot\alpha(f(a_2))) \nonumber
\end{eqnarray}
On the one hand, 
\begin{eqnarray}
 (\alpha(a_1)\cdot (f(a_2))_1)\cdot g(\alpha((f(a_2))_2))
&=&(\alpha(a_1)\cdot (f(a_2))_1)\cdot \alpha(g(f(a_2))_2)\nonumber\\
&\stackrel{(\ref{cent})}{=}&(a_1\cdot \alpha((f(a_2))_1))\cdot \alpha(g(f(a_2))_2)\nonumber\\
&\stackrel{(\ref{had})}{=}&\alpha(a_1)\cdot (\alpha((f(a_2))_1)\cdot g(f(a_2))_2)\nonumber\\
&\stackrel{}{=}&\alpha(a_1)\cdot (R(g)(f(a_2)))\nonumber\\
&=&\alpha(a_1)\cdot ((R(g)f)(a_2))\nonumber\\
&=&R(R(g)f)(a).\nonumber
\end{eqnarray}
On the other hand,
\begin{eqnarray}
 (\alpha(a_1))_1\otimes g\Big((\alpha(a_1))_2\cdot\alpha(f(a_2))\Big)
&=&\mu(id\otimes g)(id\otimes \mu)\Big((\alpha(a_1))_1\otimes (\alpha(a_1))_2\otimes \alpha(f(a_2))\Big)\nonumber\\
&\stackrel{(\ref{cocent})}{=}&\mu(id\otimes g)(id\otimes \mu)\Big(a_{11}\otimes \alpha(a_{12})\otimes \alpha(f(a_2)\Big)\nonumber\\
&\stackrel{(\ref{hca})}{=}&\mu(id\otimes g)(id\otimes \mu)(id\otimes\alpha\otimes f)(a_{11}\otimes a_{12}\otimes \alpha(a_2))\nonumber\\
&\stackrel{(\ref{hca})}{=}&\mu(id\otimes g)(id\otimes \mu)(id\otimes\alpha\otimes f)(\alpha(a_1)\otimes a_{21}\otimes a_{22})\nonumber\\
&=&\mu(id\otimes g)(id\otimes \mu)(\alpha(a_1)\otimes\alpha(a_{21})\otimes f(a_{22}))\nonumber\\
&=&\mu(id\otimes g)(\alpha(a_1)\otimes\alpha(a_{21})\cdot f(a_{22}))\nonumber\\
&=&\alpha(a_1)\cdot g(\alpha(a_{21})\cdot f(a_{22}))\nonumber\\
&=&\alpha(a_1)\cdot g(R(f)(a_{2}))= \alpha(a_1)\cdot(g(R(f)))(a_{2})\nonumber\\
&=&R(g(R(f)))(a)\nonumber.
\end{eqnarray}
Thus the conclusion holds. See also \cite{MD} for examples.
\end{exam}

 This is the Hom-version of  (\cite{EK}, section 4).
\begin{prop}
 Let $(A, \cdot, \alpha, R)$ be a Hom-associative Rota-Baxter algebra. Define
\begin{eqnarray}
 x\dashv y&:=&x\cdot R(y)-x\cdot y,\nonumber\\
x\vdash y&:=&R(x)\cdot y+x\cdot y\nonumber.
\end{eqnarray}
 Then 
$(A, \dashv, \vdash, \alpha)$ is a Hom-dendriform algebra.
\end{prop}
\begin{proof}
 For any $x, y, z\in A$, we have
\begin{eqnarray}
 (x\dashv y)\dashv\alpha(z)
&=&(x\cdot R(y)-x\cdot y)\dashv\alpha(z)\nonumber\\
&=&(x\cdot R(y))\cdot R(\alpha(z))-(x\cdot R(y))\cdot \alpha(z)-(x\cdot y)\cdot R(\alpha(z))+(x\cdot y)\cdot\alpha(z)\nonumber.
\end{eqnarray}
Using the Hom-associativity and the fact that $\alpha$ commutes with $R$,
\begin{eqnarray}
 (x\dashv y)\dashv\alpha(z)
&=&\alpha(x)\cdot( R(y)\cdot R(z))-\alpha(x)\cdot(R(y)\cdot z)-\alpha(x)\cdot(y\cdot R(z))+\alpha(x)\cdot(y\cdot z)\nonumber\\
&=&\alpha(x)\cdot\Big( R(y)\cdot R(z)-R(y)\cdot z-y\cdot R(z)+y\cdot z\Big)\nonumber.
\end{eqnarray}
$R$ being a Rota-Baxter operator and adding $(y\cdot z-z\cdot y)+(R(y\cdot z)-R(z\cdot y))$,
\begin{eqnarray}
(x\dashv y)\dashv\alpha(z)&=&\alpha(x)\cdot\Big[R\Big(R(y)\cdot z+y\cdot R(z)\Big)-y\cdot z\nonumber\\
&&\quad-R(y)\cdot z-y\cdot R(z)+y\cdot z+ R(y\cdot z)- R(y\cdot z)\Big]\nonumber\\
&=&\alpha(x)\cdot\Big[R\Big(y\cdot R(z)-y\cdot z\Big)-\Big(y\cdot R(z)-y\cdot z\Big)\Big]\nonumber\\
&&+\alpha(x)\cdot\Big[R\Big(R(y)\cdot z+y\cdot z\Big)-\Big(R(y)\cdot z+y\cdot z\Big)\Big]\nonumber\\
&=&\alpha(x)\dashv (y\cdot R(z)-y\cdot z)+\alpha(x)\dashv(R(y)\cdot z+y\cdot z)\nonumber\\
&=&\alpha(x)\dashv(y\dashv z)+\alpha(x)\dashv(y\vdash z).\nonumber
\end{eqnarray}
The other relations are proved similarly.
\end{proof}
\section{Hom-Post-Lie modules}
 In this section, we introduce modules over Hom-post-Lie algebras. We give some constructions and various twisting.
\begin{defn}\label{hplad}\cite{B2}
 A Hom-post-Lie algebra $(L, [-, -], \cdot, \alpha)$ is a Hom-Lie algebra $(L, [-, -], \alpha)$ together with 
 a bilinear map $\cdot : L\otimes L\rightarrow L$ such that :
\begin{eqnarray}
 \alpha(z)\cdot [x, y]-[z\cdot x, \alpha(y)]-[\alpha(x), z\cdot y]=0, \label{pl4}
\end{eqnarray}
\begin{eqnarray}
\alpha(z)\cdot(y\cdot x)-\alpha(y)\cdot(z\cdot x)+(y\cdot z)\cdot\alpha(x)-(z\cdot y)\cdot\alpha(x)+[y, z]\cdot\alpha(x)=0,\label{pl3}
\end{eqnarray}

for any $x, y, z\in L$.\\

If in addition, $\alpha([x, y])=[\alpha(x), \alpha(y)]$ and $\alpha(x\cdot y)=\alpha(x)\cdot\alpha(y)$, then $(L, [-, -], \cdot, \alpha)$ is 
said to be a multiplicative Hom-post-Lie algebra.
\end{defn}
\begin{defn}
 Let $(A, [-, -],  \cdot, \alpha)$ and $(A', [-, -]', \cdot', \alpha')$ be two post-Hom-Lie algebras.
 A morphism of post-Hom-Lie algebras is a linear map $f : A\rightarrow A'$ such that 
$$\alpha'\circ f=f\circ\alpha,\quad f([x, y])=[f(x), f(y)]' \quad\mbox{and}\quad f(x\cdot y)=f(x)\cdot' f(y).$$
\end{defn}
For example, the twisting map of any multiplicative post-Hom-Lie algebra is a morphism of post-Hom-Lie algebras.
\begin{rem}
The identities (\ref{pl4}) and (\ref{pl3}) can be respectively written as
\begin{eqnarray}
L_{\alpha(x)}\lambda_y-\lambda_{\alpha(y)}L_x&=&\lambda_{xy}\alpha\nonumber,\\
 L_{\alpha(x)} L_y-L_{\alpha(y)} L_x+L_{yx}\alpha-L_{xy}\alpha&=&L_{[x, y]}\alpha,\nonumber
\end{eqnarray}
where $L_x(y)=x\cdot y$ and $\lambda_x(y)=[x, y]$.
\end{rem}
\begin{exam}
 A post-Lie algebra \cite{ML} is a Hom-post-Lie algebra with $\alpha=Id$.
\end{exam}

\begin{exam}
 Any Hom-preLie algebra is a Hom-post-Lie algebra with the trivial Hom-Lie bracket. 
\end{exam}

\begin{exam}
 Any commutative Hom-associative algebra give rises to Hom-post-Lie algebra with the commutator bracket.
\end{exam}

\begin{exam}
 If $(L, [-, -], \cdot, \alpha)$ is a Hom-post-Lie algebra, then for any parameter $k\in\mathbb{K}^*$, 
$(L, [-, -]_k=k[., .], \cdot_k=k\cdot, \alpha)$ is a Hom-post-Lie algebra as well.
\end{exam}

\begin{rem}
Let $(L, [. ,.], \alpha)$ be an $n$-dimensional Hom-Lie algebra and $\{e_i\}$ be its basis. 
Let $[e_i, e_j]=c_{ij}^ke_k$ and $\alpha(e_i)=\alpha_i^je_j$.
To construct a Hom-post-Lie algebra from a Hom-Lie algebra, we should define a multiplication $\cdot : L\otimes L\rightarrow L$ such that 
axioms (\ref{pl4}) and (\ref{pl3}) be satisfied for any element of the basis. Let $e_i\cdot e_j=m_{ij}^ke_k$, then 
(\ref{pl4}) writes 
\begin{eqnarray}
&&\qquad\alpha(e_k)\cdot[e_i, e_j]-[e_k\cdot e_i, \alpha(e_j)]-[\alpha(e_i), e_k\cdot e_j]=0 \nonumber\\
&&\qquad\qquad\Longleftrightarrow \alpha_k^pc_{ij}^lm_{pl}^q-\alpha_j^pc_{lp}^qm_{ki}^l
-\alpha_i^lc_{lp}^qm_{kj}^p=0,\; i,j,k,l, m, p,q=1,\dots, n\label{pl41}
\end{eqnarray}
giving a linear system in $m_{ij}^k$ of $n^4$ equations and $n^3$ unknowns.\\
Axioms (\ref{pl3}) writes
\begin{eqnarray}
&&\alpha(e_k)\cdot(e_j\cdot e_i)-\alpha(e_j)\cdot(e_k\cdot e_i)+(e_j\cdot e_k)\cdot e_i -(e_k\cdot e_j)\cdot e_i+[e_j, e_k]\cdot\alpha(e_i)=0
 \nonumber\\
&&\qquad\Longleftrightarrow \alpha_k^lm_{ji}^pm_{lp}^q-\alpha_j^lm_{ki}^pm_{lp}^q+\alpha_i^pm_{jk}^lm_{lp}^q\nonumber\\
&&\qquad\qquad\qquad\qquad-\alpha_i^pm_{kj}^lm_{lp}^q+\alpha_i^pc_{jk}^lm_{lp}^q=0, i,j,k,l, m,p, q=1,\dots, n\label{pl42}
\end{eqnarray}
which gives a non-linear system in $m_{ij}^k$ of $n^4$ equations and $n^3$ unknowns.\\
Solving first the system (\ref{pl41}) and then checking if the solutions satisfy equations of system (\ref{pl42}), we obtain examples
of Hom-post-Lie algebras.
\end{rem}

\begin{exam}
Recall that a Hom-Novikov algebra \cite{DYN}, \cite{B3}, is a triple $(A, \cdot, \alpha)$ consisting of a vector space $A$, a bilinear map 
$\cdot : A\times A\rightarrow A$ and an endomorphism $\alpha : A\rightarrow A$ satisfying
\begin{eqnarray}
 (x\cdot y)\cdot\alpha(z)&=&(x\cdot z)\cdot\alpha(y),\label{hn1}\\
(x\cdot y)\cdot\alpha(z)-\alpha(x)(y\cdot z)&=&(y\cdot x)\cdot\alpha(z)-\alpha(y)\cdot(x\cdot z), \label{ls}
\end{eqnarray}
for all $x, y, z\in A$.\\ 
Recall also that a  Hom-algebra $(A, \cdot, \alpha)$ is said to be left commutative \cite{MRB}, if
\begin{eqnarray}
 (x\cdot y)\cdot\alpha(z)=(y\cdot x)\cdot\alpha(z) \;\;  \mbox{(left commutativity)}\nonumber
\end{eqnarray}
for any $x, y\in A$.\\
Then a Hom-Novikov algebra $(A, \cdot, \alpha)$ in which the product is left commutative carries a structure of a post-Hom-Lie
algebra with the commutator bracket. In fact, we know that $(A, [., .], \alpha)$ is a Hom-Lie algebra \cite{YY}, \cite{B3}.
The condition (\ref{pl3}) comes from (\ref{ls}) and the left commutativity. It remains to prove condition (\ref{pl4}).
For any $x, y, z\in A$, 
\begin{eqnarray}
&&\qquad\qquad \alpha(z)\cdot [x, y]-[z\cdot x, \alpha(y)]-[\alpha(x), z\cdot y]=\nonumber\\
&&=\alpha(z)\cdot(x\cdot y)-\alpha(z)\cdot(y\cdot x)-(z\cdot x)\cdot\alpha(y)+ \alpha(y)\cdot(z\cdot x)-\alpha(x)\cdot(z\cdot y)+
(z\cdot y)\cdot\alpha(x)\nonumber\\
&&\stackrel{(\ref{hn1})}{=}\alpha(z)\cdot(x\cdot y)-\alpha(z)\cdot(y\cdot x)+ \alpha(y)\cdot(z\cdot x)-\alpha(x)\cdot(z\cdot y)\nonumber\\
&&\stackrel{(\ref{ls})}{=}(z\cdot x)\alpha(y)-(z\cdot x)\alpha(y) -\alpha(z)\cdot(y\cdot x)+ \alpha(y)\cdot(z\cdot x)\nonumber.
\end{eqnarray}
By the left-commutativity, 
\begin{eqnarray}
 \alpha(z)\cdot [x, y]-[z\cdot x, \alpha(y)]-[\alpha(x), z\cdot y]
&\stackrel{(\ref{ls})}{=}&-\alpha(z)\cdot(y\cdot x)+ \alpha(y)\cdot(z\cdot x)\nonumber\\
&\stackrel{(\ref{ls})}{=}&-(x\cdot z)\alpha(y)+(z\cdot x)\alpha(y)\nonumber.
\end{eqnarray}
The left hand side vanishes by the left-commutativity.
\end{exam}

Now, we define modules over Hom-post-Lie algebras.
\begin{defn}
Let $(L, [-, -], \cdot, \alpha)$ be a Hom-post-Lie algebra. A (left)module over $L$ is a Hom-module $(M, \alpha_M)$ equipped with two
linear maps $l_\ast, l_\star: L\rightarrow\mathcal{G}l(M)$ such that :
 \begin{eqnarray}
\alpha_M l_\ast(x)&=&l_\ast(x)\alpha_M \quad\mbox{and}\quad\alpha_M l_\star(x)=l_\star(x)\alpha_M,\label{ma1}\\
  l_\star([x, y])\alpha_M&=&l_\star(\alpha(x))l_\star(y)-l_\star(\alpha(y))l_\star(x),\\
l_\star(x\cdot y)\alpha_M&=&l_\ast(\alpha(x))l_\star(y)-l_\star(\alpha(y))l_\ast(x),\\
l_\ast([x, y])\alpha_M&=&l_\ast(\alpha(x))l_\ast(y)-l_\ast(\alpha(y))l_\ast(x)-l_\ast(x\cdot y)\alpha_M+l_\ast(y\cdot x)\alpha_M,\label{ma4}
 \end{eqnarray}
for any $x, y\in L$.
\end{defn}
\begin{rem}
 When $\alpha=Id_A$ and $\alpha_M=Id_M$, we recover modules over Post-Lie algebras.
\end{rem}
\begin{exam}
 Any Hom-post-Lie algebra is a module over itself.
\end{exam}
\begin{rem}
Axioms (\ref{ma1})-(\ref{ma4}) are respectively equivalent to the following
 \begin{eqnarray}
  \alpha_M(x\diamond m)&=&\alpha(x)\diamond\alpha_M(m)\quad\mbox{and}\quad\alpha_M(x\bullet m)=\alpha(x)\bullet\alpha_M(m),\label{hpla1}\\
{ [x, y]}\diamond\alpha_M(m)&=&\alpha(x)\diamond(y\diamond m)-\alpha(y)\diamond(x\diamond m),\label{hpla2}\\
(x\cdot y)\diamond\alpha_M(m)&=&\alpha(x)\bullet(y\diamond m)-\alpha(y)\diamond(x\bullet m),\label{hpla3}\\
{ [x, y]}\bullet\alpha_M(m)&=&\alpha(x)\bullet(y\bullet m)-\alpha(y)\bullet(x\bullet m)\label{hpla4}\\
&&-(x\cdot y)\bullet\alpha_M(m)+(y\cdot x)\bullet\alpha_M(m)\nonumber,
\end{eqnarray}
for all $x, y\in L, m\in M$. Or equivalently
 \begin{eqnarray}
  \alpha_M\circ\diamond &=&\diamond\circ(\alpha_M\otimes\alpha)\quad\mbox{and}\quad\alpha_M\circ\bullet =\bullet\circ(\alpha_M\otimes\alpha),
\label{ma11}\\
\diamond\circ([., .]\otimes\alpha_M)&=&\diamond\circ(\alpha\otimes\diamond)-\diamond\circ(\alpha\otimes\diamond)\circ(\tau\otimes Id_M),
\label{ma12}\\
\diamond\circ(\cdot \otimes\alpha_M)&=&\bullet\circ(\alpha\otimes\diamond)-\diamond\circ(\alpha\otimes\bullet)\circ(\tau\otimes Id_M),\\
\bullet\circ([., .]\otimes\alpha_M)&=&\bullet\circ(\alpha\otimes\bullet)-\bullet\circ(\alpha\otimes\bullet)\circ(\tau\otimes Id_M)\nonumber\\
&&-\bullet\circ(\cdot\otimes\alpha_M)+\bullet\circ(\cdot\otimes\alpha_M)\circ(\tau\otimes Id_M),
 \end{eqnarray}
where $\tau$ is the flip map i.e. $\tau(x\otimes y)=y\otimes x$.
\end{rem}
\begin{prop}
 Let $(M_i, \diamond_i, \bullet_i, \alpha_{M_i})$ ($i=1, 2$) be two modules over the Hom-post-Lie algebra $(L, [., .], \cdot, \alpha)$.
Then the direct sum $M=M_1\oplus M_2$ is a module over $L$ for the structure maps
$$\diamond=\diamond_1\oplus\diamond_2, \quad \bullet=\bullet_1\oplus\bullet_2,\quad\alpha_M=\alpha_{M_1}\oplus\alpha_{M_2}.$$
\end{prop}
\begin{proof}
It is straightforward by calculation.
\end{proof}
The below result gives a sequence of modules from a given one.
\begin{prop}\label{nph}
  Let $(L, [-, -], \cdot, \alpha)$ be a multiplicative Hom-post-Lie algebra. Define two new operations $\{-, -\}, \ast : L\otimes L\rightarrow L$
by
$$\{x, y\}:=[\alpha^k(x), y]\quad\mbox{and}\quad x\ast y:=\alpha^k(x)\cdot y.$$
Then $(L, \{-, -\}, \ast, \alpha)$ is a module over the Hom-post-Lie algebra $(L, [-, -], \cdot, \alpha)$.
\end{prop}
\begin{proof}
The proof is straightforward. For instance, to prove axiom (\ref{hpla3}), we have for any $x, y, z\in L$, 
\begin{eqnarray}
 \{x\cdot y, \alpha(z)\}&=&[\alpha^k(x\cdot y), \alpha(z)]
\stackrel{(\ref{pl4})}{=}\alpha^{k+1}(x)\cdot[\alpha^k(y), z]-[\alpha^{k+1}(y), \alpha^k(x)\cdot z]\nonumber\\
&=&\alpha(x)\ast[\alpha^k(y), z]-\{\alpha(y), \alpha^k(x)\cdot z\}\nonumber\\
&=&\alpha(x)\ast\{y, z\}-\{\alpha(y), x\ast z\}.\nonumber
\end{eqnarray}
The other axioms are proved analogously.
\end{proof}
The next theorem asserts that the tensor product of two Hom-post-Lie modules is also another one.
\begin{thm}\label{mpm}
Let $(M_i, \diamond_i, \bullet_i, \alpha_{M_i})$ ($i=1, 2$) be two modules over the multiplicative Hom-post-Lie algebra
 $(L, [., .], \cdot, \alpha)$.
The bilinear maps $\diamond, \bullet : L\otimes M_1\otimes M_2\rightarrow M_1\otimes M_2$ and the linear map
 $\alpha_M : M_1\otimes M_2\rightarrow M_1\otimes M_2$ defined by
 \begin{eqnarray}
\alpha_M(m_1\otimes m_2)&:=&\alpha_{M_1}(m_1)\otimes\alpha_{M_2}(m_2),\nonumber\\
  x\diamond(m_1\otimes m_2)&:=&(\alpha^k(x)\diamond_1 m_1)\otimes\alpha_{M_2}(m_2)+\alpha_{M_1}(m_1)\otimes(\alpha^k(x)\diamond_2 m_2),\nonumber\\
  x\bullet(m_1\otimes m_2)&:=&(\alpha^k(x)\bullet_1 m_1)\otimes\alpha_{M_2}(m_2)+\alpha_{M_1}(m_1)\otimes(\alpha^k(x)\bullet_2 m_2)\nonumber.
 \end{eqnarray}
give to $M_1\otimes M_2$ an $L$-module structure.
\end{thm}
\begin{proof}
We first have, for any $x, y\in L, m_1\in M_1, m_2\in M_2$,
\begin{eqnarray}
 \alpha_M(x\diamond m)
&=&\alpha_M(x\diamond(m_1\otimes m_2))\nonumber\\
&=&\alpha_M\Big((\alpha^k(x)\diamond_1 m_1)\otimes\alpha_{M_2}(m_2)+\alpha_{M_1}(m_1)\otimes(\alpha^k(x)\diamond_2 m_2)\Big)\nonumber\\
&\stackrel{}{=}&(\alpha^{k+1}(x)\diamond_1 \alpha_{M_1}(m_1))\otimes\alpha_{M_2}^2(m_2)
+\alpha_{M_1}^2(m_1)\otimes(\alpha^{k+1}(x)\diamond_2 \alpha_{M_2}(m_2))\nonumber\\
&=&\alpha(x)\diamond(\alpha_{M_1}(M_1)\otimes\alpha_{M_2}(m_2))\nonumber\\
&=&\alpha(x)\diamond\alpha_{M}(m)\nonumber.
\end{eqnarray}
Next,
  \begin{eqnarray}
   \alpha(x)\bullet(y\diamond m)
&=&\alpha(x)\bullet(y\diamond(m_1\otimes m_2))\nonumber\\
&=&\alpha(x)\bullet\Big((\alpha^k(y)\diamond_1 m_1)\otimes\alpha_{M_2}(m_2)+\alpha_{M_1}(m_1)\otimes(\alpha^k(y)\diamond_2 m_2)\Big)\nonumber\\
&=&(\alpha^{k+1}(x)\bullet_1(\alpha^k(y)\diamond_1m_1))\otimes\alpha_{M_2}^2(m_2)\nonumber\\
&&+\alpha^{k+1}(y)\diamond_1\alpha_{M_1}(m_1))\otimes(\alpha^{k+1}(x)\bullet_2\alpha_{M_2}(m_2))\nonumber\\
&&+(\alpha^{k+1}(x)\bullet_1\alpha_{M_1}(m_1))\otimes(\alpha^{k+1}(y)\diamond_2\alpha_{M_2}(m_2))\nonumber\\
&&+\alpha_{M_1}^2(m_1)\otimes(\alpha^{k+1}(x)\bullet_2(\alpha^k(y)\diamond_2m_2))\nonumber,
  \end{eqnarray}
and
  \begin{eqnarray}
   \alpha(y)\diamond(x\bullet m)&=&(\alpha^{k+1}(y)\diamond_1(\alpha^k(x)\bullet_1m_1))\otimes\alpha_{M_2}^2(m_2)\nonumber\\
&&+\alpha^{k+1}(x)\bullet_1\alpha_{M_1}(m_1))\otimes(\alpha^{k+1}(y)\diamond_2\alpha_{M_2}(m_2))\nonumber\\
&&+(\alpha^{k+1}(y)\diamond_1\alpha_{M_1}(m_1))\otimes(\alpha^{k+1}(x)\bullet_2\alpha_{M_2}(m_2))\nonumber\\
&&+\alpha_{M_1}^2(m_1)\otimes(\alpha^{k+1}(y)\diamond_2(\alpha^k(x)\bullet_2m_2))\nonumber.
\end{eqnarray}

It follows that
\begin{eqnarray}
 &&\qquad\alpha(x)\bullet(y\diamond m)-\alpha(y)\diamond(x\bullet m)=\nonumber\\
&&=(\alpha^{k+1}(x)\bullet_1(\alpha^k(y)\diamond_1m_1))\otimes\alpha_{M_2}^2(m_2)
-(\alpha^{k+1}(y)\diamond_1(\alpha^k(x)\bullet_1m_1))\otimes\alpha_{M_2}^2(m_2)\nonumber\\
&&\;\;+\alpha_{M_1}^2(m_1)\otimes(\alpha^{k+1}(x)\bullet_2(\alpha^k(y)\diamond_2m_2))
-\alpha_{M_1}^2(m_1)\otimes(\alpha^{k+1}(y)\diamond_2(\alpha^k(x)\bullet_2m_2))\nonumber\\
&&=(\alpha^k(x\cdot y)\diamond_1\alpha_{M_1})\otimes\alpha_{M_2}^2(m_2)
+\alpha_{M_1}^2(m_1)\otimes(\alpha^k(x\cdot y)\diamond_2\alpha_{M_2})\nonumber\\
&&=(x\cdot y)\diamond(\alpha_{M_1}(m_1)\otimes\alpha_{M_2}(m_2))=(x\cdot y)\diamond\alpha_{M}(m).\nonumber
\end{eqnarray}
The others relations are proved in a similar way.
\end{proof}
\begin{cor}
 Let $(L, [., .], \cdot, \alpha)$ be a multiplicative Hom-post-Lie algebra. Then $(L\otimes L, \alpha\otimes\alpha)$ is an $L$-module
with the actions
 \begin{eqnarray}
  x\diamond(y\otimes z)&:=&[\alpha(x), y]\otimes\alpha(z)+\alpha(y)\otimes[\alpha(x), z],\nonumber\\
  x\bullet(y\otimes z)&:=&(\alpha(x)\cdot y)\otimes\alpha(z)+\alpha(y)\otimes(\alpha(x)\cdot z)\nonumber.
 \end{eqnarray}
\end{cor}
\begin{proof}
 It comes from Proposition \ref{nph} and Theorem \ref{mpm}.
\end{proof}

Like modules over Hom-associative algebras \cite{DYN1}, \cite{MT} we have the following twisted properties.
\begin{thm}\label{t1}
 Let $(M, \diamond, \bullet, \alpha_M)$ be a module over the multiplicative Hom-post-Lie algebra $(L, [., .], \cdot, \alpha)$. For any non-negative integer $n$,
 define
\begin{eqnarray}
 \diamond^{n,0}&:=&\diamond\circ(\alpha^n\otimes Id_M) : L\otimes M\rightarrow M,\label{ms1}\\
\bullet^{n,0}&:=&\bullet\circ(\alpha^n\otimes Id_M) : L\otimes M\rightarrow M.\label{ms2}
\end{eqnarray}
Then $(M, \diamond^{n,0}, \bullet^{n,0}, \alpha_M)$ is an $L$-module.
\end{thm}
\begin{proof}
For any positive integer $n$, we have
 \begin{eqnarray}
  \alpha_M\circ\diamond^{n,0}&=&\alpha_M\circ\diamond\circ(\alpha^n\otimes Id_M)
=\diamond\circ(\alpha\otimes\alpha_M)\circ(\alpha^n\otimes Id_M)\quad (\mbox{by}\;\; (\ref{ma11}))\nonumber\\
&=&\diamond\circ(\alpha^n\otimes Id_M)\circ(\alpha\otimes \alpha_M)=\diamond^{n,0}\circ(\alpha\otimes\alpha_M).\nonumber
 \end{eqnarray}
Next,
\begin{eqnarray}
 \diamond^{n,0}\circ([., .]\otimes\alpha_M)
&=&\diamond\circ(\alpha^n\otimes Id_M)\circ([., .]\otimes\alpha_M)\nonumber\\
&=&\diamond\circ(\alpha^n\circ[., .]\otimes\alpha_M)\nonumber\\
&=&\diamond\circ([., .]\circ(\alpha^n\otimes\alpha^n)\otimes\alpha_M)\quad(\mbox{by multiplicativity})\nonumber\\
&=&\diamond\circ([., .]\otimes\alpha_M)\circ(\alpha^n\otimes\alpha^n\otimes Id_M)\nonumber.
\end{eqnarray}
Then using (\ref{ma12}) and (\ref{ms1}) we obtain
\begin{eqnarray}
 \diamond^{n,0}\circ([., .]\otimes\alpha_M)
&=&\diamond\circ(\alpha^n\otimes Id_M)\circ([., .]\otimes\alpha_M)\nonumber\\
&=&\diamond\circ\circ([., .]\otimes\alpha_M)\circ(\alpha^n\otimes\alpha^n\otimes Id_M)\nonumber\\
&=&\Big(\diamond\circ(\alpha\otimes\diamond)
-\diamond\circ(\alpha\otimes\diamond)\circ(\tau\otimes Id_M)\Big)\circ(\alpha^n\otimes\alpha^n\otimes Id_M)\nonumber\\
&=&\diamond\circ(\alpha^{n+1}\otimes\diamond\circ(\alpha^n\otimes Id_M))
- \diamond\circ(\alpha^{n+1}\otimes\diamond\circ(\alpha^n\otimes Id_M))\circ(\tau\otimes Id_M)\nonumber\\
&=&\diamond\circ(\alpha^{n+1}\otimes\diamond^{n,0})-\diamond\circ(\alpha^{n+1}\otimes\diamond^{n,0})\circ(\tau\otimes Id_M)\nonumber\\
&=&\diamond\circ(\alpha^n\otimes Id_M)\circ(\alpha\otimes\diamond^{n,0})
-\diamond\circ(\alpha^n\otimes Id_M)\circ(\alpha\otimes\diamond^{n,0})\circ(\tau\otimes Id_M)\nonumber\\
&=&\diamond^{n,0}\circ(\alpha\otimes\diamond^{n,0})-\diamond^{n,0}\circ(\alpha\otimes\diamond^{n,0})\circ(\tau\otimes Id_M)\nonumber.
\end{eqnarray}
The others identities are similarly proved.
\end{proof}
\begin{rem}
 Whenever the product $\bullet$ is identically nul and $n=2$, we recover Lemma 2.13 of \cite{BI2}.
\end{rem}

\begin{thm}\label{t2}
 Let $(M, \diamond, \bullet, \alpha_M)$ be a module over the Hom-post-Lie algebra $(L, [., .], \cdot, \alpha)$. For any non-negative integer $k$,
 define
\begin{eqnarray}
 \diamond^{0, k}&:=&\alpha_M^{2^k-1}\circ\diamond : L\otimes M\rightarrow M,\label{ms3}\\
\bullet^{0, k}&:=&\alpha_M^{2^k-1}\circ\bullet : L\otimes M\rightarrow M.\label{ms4}
\end{eqnarray}
Then $(M, \diamond^{0,k}, \bullet^{0,k}, \alpha_M^{2^k})$ is a module over 
$(L, \alpha^{2^k-1}\circ[., .], \alpha^{2^k-1}\circ\cdot, \alpha^{2^k})$.
\end{thm}
\begin{proof}
For any positive integer $k$, we have
 \begin{eqnarray}
  \alpha_M^{2^k}\circ\diamond^{0,k}
&=&\alpha_M^{2^k}\circ\alpha_M^{2^k-1}\circ\diamond
=\alpha_M^{2^k-1}\circ\alpha_M^{2^k}\circ\diamond.\nonumber
\end{eqnarray}
By applying the first condition of (\ref{ma11}) ${2^k}$ times, we get
 \begin{eqnarray}
  \alpha_M^{2^k}\circ\diamond^{0,k}
=\alpha_M^{2^k-1}\circ\diamond\circ(\alpha^{2^k}\otimes\alpha_M^{2^k})
=\diamond^{0, k}\circ(\alpha^{2^k}\otimes\alpha_M^{2^k}).\nonumber
 \end{eqnarray}
Similarly, by the second condition of (\ref{ma11}), we prove that
\begin{eqnarray}
 \alpha_M^{2^k}\circ\bullet^{0,k}=\bullet^{0, k}\circ(\alpha^{2^k}\otimes\alpha_M^{2^k}).\nonumber
\end{eqnarray}

Next, using $({2^k-1})$ times the first condition of (\ref{ma11}),
\begin{eqnarray}
 \diamond^{0, k}\circ(\alpha^{2^k-1}\circ[., .]\otimes\alpha_M^{2^k})
&=&\alpha_M^{2^k-1}\circ\diamond\circ(\alpha^{2^k-1}\circ[., .]\otimes\alpha_M^{2^k})\nonumber\\
&=&\alpha_M^{2^k-1}\circ\diamond\circ(\alpha^{2^k-1}\otimes\alpha_M^{2^k-1})\circ([., .]\otimes\alpha_M)\nonumber\\
&=&\alpha_M^{2^k-1}\circ \alpha_M^{2^k-1}\circ\diamond\circ([., .]\otimes\alpha_M).\nonumber
\end{eqnarray}
According to (\ref{ma12}), it follows that
\begin{eqnarray}
\diamond^{0, k}\circ(\alpha^{2^k-1}\circ[., .]\otimes\alpha_M^{2^k})
&=&\alpha_M^{2(2^k-1)}\circ\Big(\diamond\circ(\alpha\otimes\diamond)-\diamond\circ(\alpha\otimes\diamond)\circ(\tau\otimes Id_M)\Big)\nonumber\\
&=&\alpha_M^{2^k-1}\circ\alpha_M^{2^k-1}\circ\diamond\circ(\alpha\otimes\diamond)\nonumber\\
&&-\alpha_M^{2^k-1}\circ\alpha_M^{2^k-1}\circ\diamond\circ(\alpha\otimes\diamond)\circ(\tau\otimes Id_M).\nonumber
\end{eqnarray}
Thanks again to the first condition of (\ref{ma11}),
\begin{eqnarray}
\diamond^{0, k}\circ(\alpha^{2^k-1}\circ[., .]\otimes\alpha_M^{2^k})
&=&\alpha_M^{2^k-1}\circ\diamond\circ(\alpha^{2^k}\otimes\alpha_M^{2^k-1}\circ\diamond)\nonumber\\
&&-\alpha_M^{2^k-1}\circ\diamond\circ(\alpha^{2^k}\otimes\alpha_M^{2^k-1}\circ\diamond)\circ(\tau\otimes Id_M)\nonumber\\
&=&\diamond^{0, k}\circ(\alpha^{2^k}\otimes\diamond^{0, k})-\diamond^{0, k}\circ(\alpha^{2^k}\otimes\diamond^{0, k})(\tau\otimes Id_M).\nonumber
\end{eqnarray}
The proofs of the rest of the conditions are analogue.
\end{proof}
From Theorem \ref{t1} and Theorem \ref{t2}, we have the following consequence.
\begin{cor}\label{c}
 Let $(M, \diamond, \bullet, \alpha_M)$ be a module over the multiplicative Hom-post-Lie algebra $(L, [., .], \cdot, \alpha)$.
 For any non-negative integer $k$, define
\begin{eqnarray}
 \diamond^{0, k}&:=&\alpha_M^{2^k-1}\circ\diamond\circ(\alpha^n\otimes Id_M) : L\otimes M\rightarrow M,\label{ms3}\\
\bullet^{0, k}&:=&\alpha_M^{2^k-1}\circ\bullet\circ(\alpha^n\otimes Id_M) : L\otimes M\rightarrow M.\label{ms4}
\end{eqnarray}
Then $(M, \diamond^{0,k}, \bullet^{0,k}, \alpha_M^{2^k})$ is a module over 
$(L, \alpha^{2^k-1}\circ[., .], \alpha^{2^k-1}\circ\cdot, \alpha^{2^k})$.
\end{cor}

Now we have the below result.
\begin{thm}
 Let $(M, \diamond, \bullet, \alpha_M)$ be a module over the Hom-post-Lie algebra $(L, [., .], \cdot, \alpha)$.
Let $\beta : L\rightarrow L$ be an endomorphism of $A$ and $\beta_M : M\rightarrow M$ be a linear map such that
$\alpha_M\circ\beta_M=\beta_M\circ\alpha_M$, $\beta_M\circ\diamond=\diamond\circ(\beta\otimes\beta_M)$ and 
$\beta_M\circ\bullet=\bullet\circ(\beta\otimes\beta_M)$.
Define
\begin{eqnarray}
\tilde\diamond:=\beta_M\circ\diamond\circ(\beta\otimes Id_M) : L\otimes M\rightarrow M,\label{ms5}\\
\tilde\bullet:=\beta_M\circ\bullet\circ(\beta\otimes Id_M) : L\otimes M\rightarrow M.\label{ms6}
\end{eqnarray}
Then $(M, \tilde\diamond, \tilde\bullet, \alpha_M\circ\beta_M)$ is a module over $(L, \beta\circ[., .], \beta\circ\cdot, \beta\circ\alpha)$.
\end{thm}
\begin{proof}
We have
\begin{eqnarray}
 (\beta_M\circ\alpha_M)\circ\tilde\diamond
&=&(\beta_M\circ\alpha_M)\circ\beta_M\circ\diamond\circ(\beta\otimes Id_M)\nonumber\\
&=&\beta_M\circ\alpha_M\circ\diamond\circ(\beta\otimes\beta_M)\circ(\beta\otimes Id_M)\nonumber\\
&=&\beta_M\circ\diamond\circ(\alpha\otimes\alpha_M)\circ(\beta\otimes\beta_M)\circ(\beta\otimes Id_M)\nonumber\\
&=&\beta_M\circ\diamond\circ(\beta\otimes Id_M)\circ(\alpha\circ\beta\otimes\alpha_M\circ\beta_M)\nonumber\\
&=&\tilde\diamond\circ(\alpha\circ\beta\otimes\alpha_M\circ\beta_M).\nonumber
\end{eqnarray}
Next,
\begin{eqnarray}
 \tilde\diamond\circ(\beta\circ[., .]\otimes\alpha_M\circ\beta_M)
&=&\beta_M\circ\diamond\circ(\beta\otimes Id_M)\circ(\beta\circ[., .]\otimes\beta_M\circ\alpha_M)\nonumber\\
&&=\beta_M\circ\diamond\circ(\beta\otimes\beta_M)\circ([., .]\otimes \alpha_M)\circ(\beta\otimes\beta\otimes Id_M)\nonumber\\
&&=\beta_M^2\circ\diamond\circ([., .]\otimes\alpha_M)\circ(\beta\otimes\beta\otimes Id_M)\nonumber.
\end{eqnarray}
Now using (\ref{ma12}), the first identity in (\ref{ma11}) and (\ref{ms5}) we obtain
\begin{eqnarray}
\tilde\diamond\circ(\beta\circ[., .]\otimes\alpha_M\circ\beta_M)
&=&\beta_M^2\circ\Big(\diamond\circ(\alpha\otimes\diamond)
-\diamond\circ(\alpha\otimes\diamond)\circ(\tau\otimes Id_M)\Big)\circ(\beta\otimes\beta\otimes Id_M)\nonumber\\
&=&\beta_M\circ\beta_M\circ\diamond\circ(\alpha\circ\beta\otimes\diamond\circ(\beta\otimes Id_M))\nonumber\\
&&\quad-\beta_M\circ\beta_M\circ\diamond\circ(\alpha\circ\beta\otimes\diamond\circ(\beta\otimes Id_M))\circ(\tau\otimes Id_M)\nonumber\\
&=&\beta_M\circ\diamond\circ(\beta\circ\alpha\circ\beta\otimes\beta_M\circ\diamond\circ(\beta\otimes Id_M)\nonumber\\
&&\quad-\beta_M\circ\diamond\circ(\beta\circ\alpha\circ\beta\otimes\beta_M\circ\diamond\circ(\beta\otimes Id_M)\circ(\tau\otimes Id_M)\nonumber\\
&=&\beta_M\circ\diamond\circ(\alpha^2\circ\beta\otimes\tilde\diamond)
-\beta_M\circ\diamond\circ(\alpha^2\circ\beta\otimes\tilde\diamond)\circ(\tau\otimes Id_M)\nonumber\\
&=&\beta_M\circ\diamond\circ(\beta\otimes Id_M)\circ(\alpha\circ\beta\otimes\tilde\diamond)\nonumber\\
&&\quad-\beta_M\circ\diamond\circ(\beta\otimes Id_M)\circ(\alpha\circ\beta\otimes\tilde\diamond)\circ(\tau\otimes Id_M)\nonumber\\
&=&\tilde\diamond\circ(\alpha\circ\beta\otimes\tilde\diamond)
-\tilde\diamond\circ(\alpha\circ\beta\otimes\tilde\diamond)\circ(\tau\otimes Id_M)\nonumber.
\end{eqnarray}
The others identities are analogously proved.
\end{proof}
\begin{rem}
 Taking $\beta_M=\alpha_M^{2^k-1}$, $\beta=\alpha^n$, we recover Corollary \ref{c}.
\end{rem}
\begin{cor}
  Let $(M, \diamond, \bullet, \alpha_M)$ be a module over the multiplicative Hom-post-Lie algebra $(L, [., .], \cdot, \alpha)$.
Let $\beta_M : M\rightarrow M$ be a linear map such that
$\alpha_M\circ\beta_M=\beta_M\circ\alpha_M$, $\beta_M\circ\diamond=\diamond\circ(\alpha\otimes\beta_M)$ and 
$\beta_M\circ\bullet=\bullet\circ(\alpha\otimes\beta_M)$.
Define
\begin{eqnarray}
\tilde\diamond:=\beta_M\circ\diamond\circ(\alpha\otimes Id_M) : L\otimes M\rightarrow M,\label{ms5}\\
\tilde\bullet:=\beta_M\circ\bullet\circ(\alpha\otimes Id_M) : L\otimes M\rightarrow M.\label{ms6}
\end{eqnarray}
Then $(M, \tilde\diamond, \tilde\bullet, \alpha_M\circ\beta_M)$ is a module over $(L, \alpha\circ[., .], \alpha\circ\cdot, \alpha^2)$.
\end{cor}
\begin{proof}
 Take $\beta=\alpha$.
\end{proof}

\begin{cor}
  Let $(M, \diamond, \bullet, \alpha_M)$ be a module over the multiplicative Hom-post-Lie algebra $(L, [., .], \cdot, \alpha)$.
Define
\begin{eqnarray}
\tilde\diamond:=\alpha_M\circ\diamond\circ(\alpha\otimes Id_M) : L\otimes M\rightarrow M,\label{ms5}\\
\tilde\bullet:=\alpha_M\circ\bullet\circ(\alpha\otimes Id_M) : L\otimes M\rightarrow M.\label{ms6}
\end{eqnarray}
Then $(M, \tilde\diamond, \tilde\bullet, \alpha^2_M)$ is a module over $(L, \alpha\circ[., .], \alpha\circ\cdot, \alpha^2)$.
\end{cor}
\begin{proof}
 Take $\beta=\alpha$ and $\beta_M=\alpha_M$.
\end{proof}

\section{Some Functors}
In this section, we introduce Hom-$L$-dendriform algebras, study various properties and establish their connection with Hom-preLie algebras.
\subsection{Hom-$L$-dendriform algebras}
\begin{defn}
{\footnotesize Let $A$ be a vector space with two binary operations denoted by $\triangleleft, \triangleright : A \otimes A\rightarrow A$ 
and $\alpha : A \rightarrow A$ a linear map. The quadruple $(A, \triangleleft, \triangleright, \alpha)$ is called an $L$-dendriform algebra if 
for any $x, y, z\in A$,
\begin{eqnarray}
 \alpha(x)\triangleright (y\triangleright z)&=&(x\triangleright y)\triangleright\alpha(z)+(x\triangleleft y)\triangleright\alpha(z)
+\alpha(y)\triangleright(x\triangleright z)-(y\triangleleft x)\triangleright\alpha(z)-(y\triangleright x)\triangleright\alpha(z),\label{lhd1} \\
\alpha(x)\triangleright(y\triangleleft z)&=&(x\triangleright y)\triangleleft\alpha(z)+\alpha(y)\triangleleft(x\triangleright z)
+\alpha(y)\triangleleft(x\triangleleft z)-(y\triangleleft x)\triangleleft\alpha(z).\label{lhd2}
\end{eqnarray}
}
\end{defn}
\begin{prop}\label{hldtopl}
 Let $(A, \triangleright, \triangleleft, \alpha)$ be a Hom-$L$-dendriform algebra.
 The binary operation $\cdot : A\otimes A\rightarrow A$ (resp. $\ast : A\otimes A\rightarrow A$) given by
\begin{eqnarray}
 x\cdot y=x\triangleright y+x\triangleleft y\; (\mbox{resp.}\; x\ast y=x\triangleright y-y\triangleleft x)\label{pl1}
\end{eqnarray}
defines a Hom-preLie algebra. The triple $(A, \cdot, \alpha)$ (resp. $(A, \ast, \alpha)$) is called the associated horizontal (resp. vertical)
 Hom-preLie algebra of $(A, \triangleright, \triangleleft, \alpha)$.
\end{prop}
Whenever $\alpha=Id_A$, we get :
\begin{cor}\cite{CBX}
 Let $(A, \triangleright, \triangleleft)$ be an $L$-dendriform algebra.
 The binary operation $\cdot : A\otimes A\rightarrow A$ (resp. $\ast : A\otimes A\rightarrow A$) given by
\begin{eqnarray}
 x\cdot y=x\triangleright y+x\triangleleft y\; (\mbox{resp.}\; x\ast y=x\triangleright y-y\triangleleft x)\label{pl1}
\end{eqnarray}
defines a preLie algebra. The triple $(A, \cdot)$ (resp. $(A, \ast)$) is called the associated horizontal (resp. vertical)
 preLie algebra of $(A, \triangleright, \triangleleft)$.
\end{cor}
The following corollary comes from Lemma \ref{pltol} and Proposition \ref{hldtopl}.
\begin{cor}
 Let $(A, \triangleright, \triangleleft, \alpha)$ be a Hom-$L$-dendriform algebra. Then each of the brackets
\begin{eqnarray}
 [x, y]&=&x\cdot y-y\cdot x=x\triangleright y+x\triangleleft y-y\triangleright x+y\triangleleft x,\\
\{x, y\}&=&x\ast y-y\ast x=x\triangleright y-y\triangleleft x-y\triangleright x+x\triangleleft y,
\end{eqnarray}
defines a Hom-Lie algebra structure on $A$.
\end{cor}

The below result allows to obtain a Hom-$L$-dendriform algebra from a given one by transposition.
\begin{prop}
 Let $(A, \triangleright, \triangleleft, \alpha)$ be a Hom-$L$-dendriform algebra. Define two binary operations $\triangleright^t, \triangleleft^t
: A\otimes A\rightarrow A$ by 
\begin{eqnarray}
 x\triangleright^ty:=x\triangleright y,\quad x\triangleleft^t y:=-y\triangleleft x
\end{eqnarray}
Then $(A, \triangleright^t, \triangleleft^t, \alpha)$ is a Hom-$L$-dendriform algebra, and $\cdot^t=\ast$ and $\ast^t=\cdot$.\\
The Hom-$L$-dendriform algebra $(A, \triangleright^t, \triangleleft^t, \alpha)$ is called the transpose of 
$(A, \triangleright, \triangleleft, \alpha)$.
\end{prop}
Observe that $\cdot$ and $\ast$ are involutives. \\
Let us define modules over Hom-$L$-dendriform algebras.
\begin{defn}
Let $(A, \triangleleft, \triangleright, \alpha)$ a Hom-$L$-dendriform algebra and $M$ a vector space.
 Let $l_\triangleleft, r_\triangleleft, l_\triangleright r_\triangleright : A\rightarrow \mathcal{G}l(M)$ be four linear maps. 
$(M, l_\triangleleft, r_\triangleleft, l_\triangleright r_\triangleright, \beta)$ is called an
$A$-bimodule if for any $x, y\in A$, the following five equations hold :
\begin{eqnarray}
 l_\triangleright([x, y])\beta&=&l_\triangleright(\alpha(x))l_\triangleright(y)-l_\triangleright(\alpha(y))l_\triangleright(x),\\
 l_\triangleright(x\ast y)\beta&=&l_\triangleright(\alpha(x))l_\triangleleft(y)
-l_\triangleleft(\alpha(y))l_\triangleright(x)-l_\triangleright(\alpha(y))l_\triangleright(x),\\
 r_\triangleright(x\triangleright y)\beta
&=&r_\triangleright(\alpha(y))r_\triangleright(x)+r_\triangleright(\alpha(y))r_\triangleleft(x)
 +l_\triangleright(\alpha(x))r_\triangleright(y),\nonumber\\
 &&-r_\triangleright(\alpha(y))l_\triangleright(x)-r_\triangleright(\alpha(y))l_\triangleleft(x),\\
 r_\triangleright(x\triangleleft y)\beta&=&r_\triangleleft(\alpha(y))r_\triangleright(x)+l_\triangleleft(\alpha(x))r_\triangleright(y)
 +l_\triangleright(\alpha(x))r_\triangleright(y)-r_\triangleleft(\alpha(y))r_\triangleleft(x),\\
r_\triangleright(x\star y)\beta&=&l_\triangleright(\alpha(x))r_\triangleleft(y)-r_\triangleleft(\alpha(x))l_\triangleright(y)
+r_\triangleleft(\alpha(y))r_\triangleleft(x).
\end{eqnarray}
\end{defn}
\begin{prop}
 Let  $A$ be a vector space with two binary operations denoted $\triangleright, \triangleleft : A\otimes A\rightarrow A$ and 
$\alpha : A\rightarrow A$ be a linear map. The quadruple $(A, \triangleright, \triangleleft, \alpha)$ is a Hom-$L$-dendriform algebra if and only
 if $(A, \cdot, \alpha)$ (resp. $(A, \ast, \alpha)$) defined by equation (\ref{pl1}) is a Hom-preLie algebra and $(A, l_\triangleright, 
r_\triangleleft, \alpha)$ (resp.   $(A, l_\triangleright, -l_\triangleleft, \alpha)$) is a module.
\end{prop}
In the next theorem, we give a construction of Hom-$L$-dendriform algebra on a direct sum via module over the given Hom-$L$-dendriform algebra.
\begin{thm}
 $(M, l_\triangleleft, r_\triangleleft, l_\triangleright r_\triangleright, \beta)$ is a bimodule over  a Hom-$L$-dendriform algebra 
$(A, \triangleleft, \triangleright, \alpha)$ if and only if the direct sum $A\oplus M$ of the underlying vector spaces of $A$ and $M$ is turned 
into a Hom-$L$-dendriform algebra by defining the twisting map $\eta : A\oplus M\rightarrow A\oplus M$ by
$$\eta(x+m) :=\alpha(x)+\beta(m)$$
and the multiplication in $A\oplus M$ by
\begin{eqnarray}
 (x+m)\vdash(y+n)&:=& x\triangleright y+l_\triangleright(x)n+r_\triangleright(y)m,\nonumber\\
 (x+m)\dashv(y+n)&:=& x\triangleleft y+l_\triangleleft(x)n+r_\triangleleft(y)m\nonumber,
\end{eqnarray}
for all $x, y\in A, m, n\in M$.
\end{thm}
\begin{proof}
It is straightforward.
\end{proof}
\subsection{Some Functors}

\begin{defn}
Let $(A, [- ,-], \alpha)$ be a Hom-Lie algebra and $(V, \beta)$ be a Hom-module. The linear map $\rho : L\rightarrow\mathcal{G}l(V)$ is called a representation of $L$ on $V$ if
 \begin{eqnarray}
 \rho([x, y])\beta =\rho(\alpha(x))\rho(y)-\rho(\alpha(y))\rho(x), \label{hlr}
\end{eqnarray}
for any $x, y\in A$.
\end{defn}
\begin{defn}
Let $(L, [- ,-], \alpha)$ be a Hom-Lie algebra and $(V, \rho, \beta)$ be a representation of $L$. A linear map $T : V\rightarrow L$ is called an 
$\mathcal{O}$-operator associated to $\rho$ if 
\begin{eqnarray}
 [T(u), T(v)])=T(\rho(T(u)v)-\rho(T(v)u)), \label{oprtl}
\end{eqnarray}
for any $u, v\in V$.
\end{defn}

From Hom-Lie algebras to Hom-preLie algebras.
\begin{thm}
 Let $(L, [-, -], \alpha)$ be a Hom-Lie algebra, $\rho : L\rightarrow\mathcal{G}l(V)$ a representation of $L$ on a vector space $V$. 
 Then the multiplication $\ast : V\times V\rightarrow V$ given by
$$u\ast v=\rho(T(u))v,$$
defines a Hom-preLie structure on $V$,
where $T : V\rightarrow L$ is an $\mathcal{O}$-operator of $L$.
\end{thm}
\begin{proof}
 For any $u, v, w\in V$, we have
\begin{eqnarray}
 as_\ast(u, v, w)
&=&(u\ast v)\ast\beta(w)-\beta(u)\ast(v\ast w)\nonumber\\
&=&(\rho(T(u)v))\ast\beta(w)-\beta(u)\ast(\rho(T(v)w))\nonumber\\
&=&\rho(T(\rho(T(u)v)))\beta(w)-\rho(T\beta(u))\rho(T(v))w\nonumber.
\end{eqnarray}
Then by  (\ref{hlr}),
\begin{eqnarray}
&&\quad as_\ast(u, v, w)-as_\ast(v, u, w)=\nonumber\\
&&=\rho(T(\rho(T(u)v)))\beta(w)-\rho(T\beta(u))\rho(T(v))w-\rho(T(\rho(T(v)u)))\beta(w)+\rho(T\beta(v))\rho(T(u))w\nonumber\\
&&=\rho(T(\rho(T(u)v)))\beta(w)-\rho(\alpha(T(u)))\rho(T(v))w-\rho(T(\rho(T(v)u)))\beta(w)+\rho(\alpha(T(v)))\rho(T(u))w\nonumber\\
&&=\rho[T(\rho(T(u)v))-T(\rho(T(v)u)))]\beta(w)-\rho([T(u), T(v)])\beta(w)\nonumber.
\end{eqnarray}
The last line vanishes by (\ref{oprtl}).
\end{proof}
\begin{cor}\cite{MD}
 Let $(L, [-, -], \alpha, R)$ be a Hom-Lie Rota-Baxter algebra where $R$ is a Rota-Baxter operator of weight $0$. Assume that $\alpha$ and $R$
commute. We define the operation $\ast$ on $L$ by
$$x\ast y=[R(x), y].$$
Then $(L, \ast, \alpha)$ is a Hom-preLie algebra.
\end{cor}

\begin{defn}
 Let $(A, \cdot, \alpha)$ be a Hom-associative algebra and $(V, l, r, \beta)$ be an $A$-bimodule. A linear map $T : V\rightarrow A$ is called an 
$\mathcal{O}$-operator associated to $(V, l, r, \beta)$ if $\alpha T=T\beta$ and
\begin{eqnarray}
 T(u)\cdot T(v)=T(l(T(u))v+r(T(v))u), \;u, v\in V. \label{oprta}
\end{eqnarray}
In particular, an $\mathcal{O}$-operator associated to the bimodule $(A, l_\cdot, r_\cdot, \alpha)$ is called a Rota-Baxter operator
 of weight $0$ on $(A, \cdot, \alpha)$.
\end{defn}
From Hom-associative algebra to Hom-dendriform algebra.
\begin{thm}
 Let $(V, l, r, \beta)$ be a bimodule over  a Hom-associative algebra $(A, \cdot, \alpha)$. If $T: V\rightarrow A$ is an $\mathcal{O}$-operator 
associated to $(V, l, r, \beta)$, then $(V, \dashv, \vdash, \beta)$ is a Hom-dendriform algebra with
$$u\dashv v:=r(T(v)u), \quad u\vdash v:=l(T(u))v,$$
for any $u, v\in V$.
\end{thm}
\begin{proof}
  For any $u, v, w\in V$, we have
\begin{eqnarray}
 (u\vdash v)\dashv\beta(w)-\beta(u)\vdash(v\dashv w)
&=&l(T(u)v)\dashv\beta(w)-\beta(u)\vdash r(T(w)v)\nonumber\\
&=&r(T\beta(w))l(T(u))v-l(T\beta(u)v))r(T(w)v)\nonumber\\
&=&r(\alpha(T(w)))l(T(u))v-l(\alpha(T(u)))r(T(w))v\nonumber.
\end{eqnarray}
Which vanishes by the second condition of (\ref{ahm}). The two other axioms  (\ref{t1}) and (\ref{t3}) are checked similarly.
\end{proof}
\begin{cor}\cite{MD}
 Let $(A, [-, -], \alpha, R)$ be a Hom-associative Rota-Baxter algebra where $R$ is a Rota-Baxter operator of weight $0$. Assume that $\alpha$ and $R$
commute. We define the operation $\dashv, \vdash$ on $A$ by
$$x\dashv y=x\cdot R(y)\quad\mbox{and}\quad x\vdash y=R(x)\cdot y.$$
Then $(A), \dashv, \vdash, \alpha)$ is a Hom-dendriform algebra.
\end{cor}
From Hom-associative algebras to Hom-preLie algebras.
\begin{thm}
 Let $(V, l, r, \beta)$ be a bimodule over  a Hom-associative algebra $(A, \cdot, \alpha)$ and $T : V\rightarrow A$ an $\mathcal{O}$-operator 
associated to $(V, l, r, \beta)$. Let us define the bilinear map $\ast : V\otimes V\rightarrow V$, for any $u, v\in V$, by
$$u\ast v:=l(T(u))v-r(T(u))v.$$
Then $(V, \ast, \beta)$ is a Hom-preLie.
\end{thm}
\begin{proof}
 On the one hand, for any $u, v, w\in V$,
\begin{eqnarray}
 \beta(u)\ast(v\ast w)&=&\beta(u)\ast(l(T(v)w)- r(T(v)w))\nonumber\\
&=&l(T\beta(u))\Big(l(T(v))w- r(T(v))w\Big)-r(T\beta(u))\Big(l(T(v))w- r(T(v))w\Big)\nonumber\\
&=&l(\alpha(T(u)))l(T(v))w - l(\alpha(T(u)))r(T(v))w\nonumber\\
&&-r(\alpha(T(u)))l(T(v))w+r(\alpha(T(u)))r(T(v))w\nonumber.
\end{eqnarray}
On the other hand, 
\begin{eqnarray}
 (u\ast v)\ast\beta(w)
&=&(l(T(u))v- r(T(u)v))\ast\beta(w)\nonumber\\
&=&l[T(l(T(u))v)- T(r(T(u)v))]\beta(w)-r[T(l(T(u))v)- T(r(T(u)v))]\beta(w)\nonumber\\
&=&l(T(l(T(u))v))\beta(w)- l(T(r(T(u)v)))\beta(w)\nonumber\\
&&-r(T(l(T(u))v))\beta(w)+r(T(r(T(u)v)))\beta(w)\nonumber.
\end{eqnarray}
Then
\begin{eqnarray}
 &&\quad as_\ast(u, v, w)-as_\ast(v, u, w)=\nonumber\\
&&=l(\alpha(T(u)))l(T(v))w - l(\alpha(T(u)))r(T(v))w-r(\alpha(T(u)))l(T(v))w+r(\alpha(T(u)))r(T(v))w\nonumber\\
&&\quad-l(T(l(T(u))v))\beta(w)+ l(T(r(T(u))v))\beta(w)+r(T(l(T(u))v))\beta(w)- r(T(r(T(u))v))\beta(w)\nonumber\\
&&-l(\alpha(T(v)))l(T(u))w + l(\alpha(T(v)))r(T(u))w+r(\alpha(T(v)))l(T(u))w-r(\alpha(T(v)))r(T(u))w\nonumber\\
&&\quad+l(T(l(T(v))u))\beta(w)- l(T(r(T(v))u))\beta(w)-r(T(l(T(v))u))\beta(w)+ r(T(r(T(v))u))\beta(w)\nonumber.
\end{eqnarray}
According to (\ref{oprta}), the second and last lines give
\begin{eqnarray}
 &&\quad as_\ast(u, v, w)-as_\ast(v, u, w)=\nonumber\\
&&=l(\alpha(T(u)))l(T(v))w - l(\alpha(T(u)))r(T(v))w-r(\alpha(T(u)))l(T(v))w+r(\alpha(T(u)))r(T(v))w\nonumber\\
&&-l(T(u)\cdot T(v))\beta(w)+l(T(v)\cdot T(u))\beta(w)+r(T(u)\cdot T(v))\beta(w)-r(T(v)\cdot T(u))\beta(w) \nonumber\\
&&-l(\alpha(T(v)))l(T(u))w + l(\alpha(T(v)))r(T(u))w+r(\alpha(T(v)))l(T(u))w-r(\alpha(T(v)))r(T(u))w\nonumber.
\end{eqnarray}
The left hand side vanishes by axioms (\ref{ahm}).
\end{proof}
\begin{cor}\cite{MD}
 Let $(A, [-, -], \alpha, R)$ be a Hom-associative Rota-Baxter algebra where $R$ is a Rota-Baxter operator of weight $0$. Assume that $\alpha$ 
and $R$ commute.  We define the operation $\ast$ on $A$ by
$$x\ast y=R(x)\cdot y-y\cdot R(x).$$
Then $(A, \ast, \alpha)$ is a Hom-preLie algebra.
\end{cor}
From Hom-associative algebras to Hom-L-dendriform algebras.
\begin{thm}
Let $(A, \cdot, \alpha)$ be a Hom-associative algebra, $(V, l, r, \beta)$ an $A$-bimodule and $T : V\rightarrow A$ an $\mathcal{O}$-operator
of $A$. Then there exists a Hom-$L$-dendriform algebra structure on $V$ defined by
$$u\triangleright v:=l(T(u))v, \quad u\triangleleft v:=r(T(v))u,$$
for any $u, v\in V$.
\end{thm}
\begin{proof}
 We have to prove axioms (\ref{lhd1}) and (\ref{lhd2}). For all $u, v, w\in V$,
\begin{eqnarray}
&&(u\triangleright y)\triangleright\beta(w)+(u\triangleleft v)\triangleright\beta(w)+\beta(v)\triangleright(u\triangleright w)\nonumber\\
&&\qquad\qquad\qquad\qquad\qquad-(v\triangleleft u)\triangleright\beta(w)-(v\triangleright u)\triangleright\beta(w)
-\beta(u)\triangleright (v\triangleright w)=\nonumber\\
&&=l(T(l(T(u))v))\beta(w)+l(T(r(T(v))u))\beta(w)+l(T\beta(v))l(T(u))w\nonumber\\
&&\quad-l(T(r(T(u))v))\beta(w)-l(T(l(T(v))u))\beta(w)-l(T\beta(u))l(T(v))w \nonumber\\
&&=l[T(l(T(u))v)+T(r(T(v))u)]\beta(w)+l(T\beta(v))l(T(u))w\nonumber\\
&&\quad -l[T(r(T(u))v)+T(l(T(v))u)]\beta(w)-l(T\beta(u))l(T(v))w\nonumber\\
&&\stackrel{(\ref{oprta})}{=}l(T(u)\cdot T(v))\beta(w)+l(\alpha(T(v)))l(T(u))w -l(T(v)\cdot T(u))\beta(w)-l(\alpha(T(u)))l(T(v))w\nonumber.
\end{eqnarray}
The left last line vanishes by the first axiom in (\ref{ahm}). The second condition is proved similarly.
\end{proof}
Taking $\alpha=Id_A$ and $\beta=Id_M$, we obtain the following corollary.
\begin{cor}\cite{CBX}
  Let $(A, \cdot)$ be an associative algebra, $(V, l, r)$ be an $A$-bimodule and $T : V\rightarrow A$ an $\mathcal{O}$-operator
of $A$. Then there exists an $L$-dendriform algebra structure on $V$ defined by
$$u\triangleright v:=l(T(u))v, \quad u\triangleleft v:=r(T(v))u,$$
for any $u, v\in V$.
\end{cor}
From Hom-preLie algebras to Hom-$L$-dendriform algebras.
\begin{defn}
 Let $(A, \cdot, \alpha)$ be a Hom-preLie algebra and $(V, l, r, \beta)$ be an $A$-bimodule. A linear map $T : V\rightarrow A$ is called an 
$\mathcal{O}$-operator associated to $(V, l, r, \beta)$ if $\alpha T=T\beta$ and
\begin{eqnarray}
 T(u)\cdot T(v)=T(l(T(u))v+r(T(v))u), \;u, v\in V. \label{oprtpl}
\end{eqnarray}
In particular, an $\mathcal{O}$-operator associated to the bimodule $(A, l_\cdot, r_\cdot, \alpha)$ is called a Rota-Baxter operator
 of weight $0$ on $(A, \cdot, \alpha)$.
\end{defn}

\begin{thm}
 Let $(A, \cdot, \alpha)$ be a Hom-preLie algebra, $(V, l, r, \beta)$ be a bimodule on $A$ and $T : V\rightarrow A$ be an $\mathcal{O}$-operator 
associated to $(V, l, r, \beta)$. Then $(V, \triangleleft, \triangleright, \beta)$ is a Hom-dendriform algebra with respect to the operations
$\triangleleft , \triangleright : V\otimes V\rightarrow V$
 defined by
$$u\triangleleft y:=l(T(u))v, \quad u\triangleright v:=-r(T(u))v$$
for any $u, v\in V$.
\end{thm}
\begin{proof}
 Hier also, we have to prove axioms (\ref{lhd1}) and (\ref{lhd2}). For all $u, v, w\in V$,
\begin{eqnarray}
&&\qquad-\beta(u)\triangleright(v\triangleleft w)+(u\triangleright v)\triangleleft\beta(w)+\beta(v)\triangleleft(u\triangleright w)
+\beta(v)\triangleleft(u\triangleleft w)-(v\triangleleft u)\triangleleft\beta(w)=\nonumber\\
&&= l(T\beta(u))r(T(v))w-r(T(l(T(u))v))\beta(w)-r(T\beta(v))l(T(u))w\nonumber\\
&&\qquad +r(T\beta(v))r(T(u))w-r(T(r(T(v))u))\beta(w)\nonumber\\
&&=l(T\beta(u))r(T(v))w-r\Big[T\Big(l(T(u))v+r(T(v))u\Big)\Big]\beta(w)\nonumber\\
&&\qquad-r(T\beta(v))l(T(u))w+r(T\beta(v))r(T(u))w\nonumber\\
&&=l(\alpha(T(u)))r(T(v))w-r(T(u)\cdot T(v))\beta(w)-r(\alpha(T(v)))l(T(u))w+r(\alpha(T(v)))r(T(u))w\nonumber.
\end{eqnarray}
The left hand side vanishes by (\ref{hlsm2}). The other condition has a similar proof.
\end{proof}
\begin{cor}
  Let $(A, \cdot)$ be a preLie algebra, $(V, l, r)$ be a bimodule on $A$ and $T : V\rightarrow A$ be an $\mathcal{O}$-operator 
associated to $(V, l, r)$. Then $(V, \triangleleft, \triangleright)$ is a dendriform algebra with respect to the operations
$\triangleleft , \triangleright : V\otimes V\rightarrow V$
 defined by
$$u\triangleleft v:=l(T(u))v, \quad u\triangleright v:=-r(T(u))v,$$
for any $u, v\in V$.
\end{cor}
\section*{Further discussion}
It is natural to study the connections with the analogues of Hom-Yang-Baxter equations for these algebras.
A similar analysis may be made for Hom-algebras with three operations (Hom-trialgebras and Hom-tridendriform algebras) or 
four operations (Hom-quadri-algebras and Hom-$L$-quadri-algebras). Their graded versions may also be studied.

\providecommand{\bysame}{\leavevmode\hbox to3em{\hrulefill}\thinspace}
\providecommand{\MR}{\relax\ifhmode\unskip\space\fi MR }
\providecommand{\MRhref}[2]{%
  \href{http://www.ams.org/mathscinet-getitem?mr=#1}{#2}
}
\providecommand{\href}[2]{#2}

\end{document}